\documentclass[11pt]{amsart}%
\usepackage{palatino}
\usepackage{mathpazo}
\usepackage{amsfonts}
\usepackage{amsmath}
\usepackage{amssymb,latexsym}
\usepackage{graphicx}
\usepackage[mathscr]{eucal}
\usepackage{amssymb}%
\providecommand{\U}[1]{\protect \rule{.1in}{.1in}}

\newtheorem{theorem}{Theorem}[section]

\newtheorem{corollary}[theorem]{Corollary}

\newtheorem{lemma}[theorem]{Lemma}

\theoremstyle{remark}
\newtheorem{remark}[theorem]{Remark}

\numberwithin{equation}{section}
\begin{document}
\title[existence of pre-modular forms]{The geometry of generalized Lam\'{e} equation, II: Existence of pre-modular forms and application}
\author{Zhijie Chen}
\address{Department of Mathematical Sciences, Yau Mathematical Sciences Center,
Tsinghua University, Beijing, 100084, China }
\email{zjchen2016@tsinghua.edu.cn}
\author{Ting-Jung Kuo}
\address{Department of Mathematics, National Taiwan Normal University, Taipei 11677, Taiwan }
\email{tjkuo1215@ntnu.edu.tw}
\author{Chang-Shou Lin}
\address{Taida Institute for Mathematical Sciences (TIMS), Center for Advanced Study in
Theoretical Sciences (CASTS), National Taiwan University, Taipei 10617, Taiwan }
\email{cslin@math.ntu.edu.tw}

\begin{abstract}
In this paper, the second in a series, we continue to study the
generalized Lam\'{e} equation with the Treibich-Verdier potential
\begin{equation*}
y^{\prime \prime }(z)=\bigg[  \sum_{k=0}^{3}n_{k}(n_{k}+1)\wp(z+\tfrac{
\omega_{k}}{2}|\tau)+B\bigg]  y(z),\quad n_{k}\in \mathbb{Z}_{\geq0}
\end{equation*}
from the monodromy aspect.
We prove the existence of a pre-modular form $Z_{r,s}^{\mathbf{n}}(\tau)$ of weight $\frac{1}{2}\sum n_k(n_k+1)$ such that the monodromy data $(r,s)$ is characterized by $Z_{r,s}^{\mathbf{n}}(\tau)=0$. This generalizes the result in \cite{LW2}, where the Lam\'{e} case (i.e. $n_1=n_2=n_3=0$) was studied by Wang and the third author. As applications, we prove among other things that the following two mean field equations
\[\Delta u+e^u=16\pi\delta_{0}\quad\text{and}\quad \Delta u+e^u=8\pi\sum_{k=1}^3\delta_{\frac{\omega_k}{2}}\]
on a flat torus $E_{\tau}:=\mathbb{C}/(\mathbb{Z}+\mathbb{Z}\tau)$ has the same number of even solutions. This result is quite surprising from the PDE point of view.
\end{abstract}

\maketitle

\section{Introduction}

Throughout the paper, as in Part I \cite{CKL1}, we use the notations $\omega_{0}=0$, $\omega_{1}=1$,
$\omega_{2}=\tau$, $\omega_{3}=1+\tau$ and $\Lambda_{\tau}=\mathbb{Z+Z}\tau$,
where $\tau \in \mathbb{H}=\{  \tau|\operatorname{Im}\tau>0\}  $.
Define $E_{\tau}:=\mathbb{C}/\Lambda_{\tau}$ to be a flat torus and $E_{\tau}[2]:=\{ \frac{\omega_{k}}{2}|k=0,1,2,3\}+\Lambda
_{\tau}$ to be the set consisting of the lattice points and $2$-torsion points
in $E_{\tau}$. For $z\in\mathbb{C}$ we denote $[z]:=z \ (\text{mod}\ \Lambda_{\tau}) \in E_{\tau}$.
For a point $[z]$ in $E_{\tau}$ we often write $z$ instead of $[z]$ to
simplify notations when no confusion arises.

Let $\wp(z)=\wp(z|\tau)$ be the
Weierstrass elliptic function with periods $\Lambda_{\tau}$ and define $e_k(\tau):=\wp(\frac{\omega_k}{2}|\tau)$, $k=1,2,3$. Let $\zeta(z)=\zeta(z|\tau):=-\int^{z}\wp(\xi|\tau)d\xi$
be the Weierstrass zeta function with two quasi-periods $\eta_{k}(\tau)$, $%
k=1,2$:
\begin{equation}
\eta_{k}(\tau):=2\zeta(\tfrac{\omega_{k}}{2} |\tau)=\zeta(z+\omega_{k} |\tau)-\zeta(z|\tau),\quad k=1,2,
\label{40-2}
\end{equation}
and $\sigma(z)=\sigma(z|\tau)$ be the Weierstrass sigma function defined by $%
\sigma(z):=\exp \int^{z}\zeta(\xi)d\xi$. Notice that $\zeta(z)$ is an odd
meromorphic function with simple poles at $\Lambda_{\tau}$ and $\sigma(z)$
is an odd entire function with simple zeros at $\Lambda_{\tau}$.

This is the second in a series of papers to study the generalized Lam\'{e} equation (denoted by H$(\mathbf{n},B,\tau)$)
\begin{equation}  \label{eq21}
y^{\prime \prime }(z)=I_{\mathbf{n}}(z;B,\tau)y(z),\quad z\in\mathbb{C},
\end{equation}
with
\begin{equation}
I_{\mathbf{n}}(z;B,\tau):=\sum_{k=0}^3n_k(n_k+1)\wp(z+\tfrac{\omega_k}{2}|\tau)+B,
\end{equation}
where $\mathbf{n}=(n_0, n_1, n_2, n_3)$ with $n_k\in\mathbb{Z}_{\geq 0}$ and $\max n_k\geq 1$. By changing variable $z\to z+\frac{\omega_k}{2}$ if necessary, we always assume $n_0\geq 1$.

H$(\mathbf{n},B,\tau)$ is the elliptic form of the well-known Heun's equation and the potential $\sum_{k=0}^3n_k(n_k+1)\wp(z+\tfrac{\omega_k}{2}|\tau)$
is the so-called \emph{Treibich-Verdier potential }(\cite{TV}), which is known
as an algebro-geometric finite-gap potential associated with the stationary
KdV hierarchy \cite{GW1,TV}. See also a
series of papers \cite{Takemura1,Takemura2,Takemura3,Takemura4,Takemura5} by Takemura,
where H$(\mathbf{n},B,\tau)$ was studied as the eigenvalue problem for the
Hamiltonian of the $BC_{1}$ (one particle) Inozemtsev model.
When $\mathbf{n}=(n,0,0,0)$, the potential $n(n+1)\wp(z|\tau)$ is the well-known
Lam\'{e} potential and (\ref{eq21}) becomes the Lam\'{e} equation
\begin{equation}  \label{Lame}
y^{\prime \prime }(z)=[n(n+1)\wp(z|\tau)+B]y(z),\quad z\in\mathbb{C}.
\end{equation}
Ince \cite{Ince} first discovered that the Lam\'{e} potential is a
finite-gap potential. See also the classic texts \cite{Halphen,Poole,Whittaker-Watson} and recent works \cite{CLW,Dahmen,LW2,Maier} for more details about (\ref{Lame}).

In this paper, we continue our study, initiated in Part I \cite{CKL1}, on H$(\mathbf{n}, B, \tau)$ from the monodromy aspect. Since the local exponents of H$(\mathbf{n}, B, \tau)$ at $\frac{\omega_{k}}{2}$ are $-n_{k}$, $n_{k}+1$ and $I_{\mathbf{n}}(z; B, \tau)$ is even elliptic, it is easily seen (cf. \cite{GW1,Takemura1}) that any solution is meromorphic in $\mathbb{C}$, i.e. the local monodromy matrix
at $\frac{\omega_{k}}{2}$ is $I_{2}$. Thus the monodromy representation is a group homeomorphism $\rho:\pi_{1}(E_{\tau})  \to
SL(2,\mathbb{C})$, which is abelian and hence reducible. Let $\ell_{j}$, $j=1,2$, be two
fundamental cycles of $E_{\tau}$. Then there are two cases (see Part I \cite{CKL1}):

\begin{itemize}
\item[(a)] Completely reducible (i.e. all the monodromy matrices have two
linearly independent common eigenfunctions). Up to a common conjugation,
$\rho(\ell_{1})$ and $\rho(\ell_{2})$ can be expressed as%
\begin{equation}
\rho(\ell_{1})=%
\begin{pmatrix}
e^{-2\pi is} & 0\\
0 & e^{2\pi is}%
\end{pmatrix}
,\text{ \  \  \ }\rho(\ell_{2})=%
\begin{pmatrix}
e^{2\pi ir} & 0\\
0 & e^{-2\pi ir}%
\end{pmatrix}
\label{Mono-1}%
\end{equation}
for some $(r,s)\in \mathbb{C}^{2}\backslash \frac{1}{2}\mathbb{Z}^{2}$. In
particular,
\begin{equation}
(\text{tr}\rho(\ell_{1}),\text{tr}\rho(\ell_{2}))=(2\cos2\pi s,2\cos2\pi
r)\not \in \{ \pm(2,2),\pm(2,-2)\}. \label{complete-rs}%
\end{equation}

\item[(b)] Not completely reducible (i.e. the space of common eigenfunctions
is of dimension $1$). Up to a common conjugation, $\rho(\ell_{1})$ and
$\rho(\ell_{2})$ can be expressed as%
\begin{equation}
\rho(\ell_{1})=\varepsilon_{1}%
\begin{pmatrix}
1 & 0\\
1 & 1
\end{pmatrix}
,\text{ \  \  \ }\rho(\ell_{2})=\varepsilon_{2}%
\begin{pmatrix}
1 & 0\\
\mathcal{C} & 1
\end{pmatrix}
, \label{Mono-21}%
\end{equation}
where $\varepsilon_{1},\varepsilon_{2}\in \{ \pm1\}$ and $\mathcal{C}%
\in \mathbb{C}\cup \{ \infty \}$. In particular,
\begin{equation}
(\text{tr}\rho(\ell_{1}),\text{tr}\rho(\ell_{2}))=(2\varepsilon_{1}%
,2\varepsilon_{2})\in \{ \pm(2,2),\pm(2,-2)\}. \label{notcompleteC}%
\end{equation}
Remark that if $\mathcal{C}=\infty$, then (\ref{Mono-21}) should be understood
as%
\begin{equation}
\rho(\ell_{1})=\varepsilon_{1}%
\begin{pmatrix}
1 & 0\\
0 & 1
\end{pmatrix}
,\text{ \  \  \ }\rho(\ell_{2})=\varepsilon_{2}%
\begin{pmatrix}
1 & 0\\
1 & 1
\end{pmatrix}
. \label{Mono-31}%
\end{equation}

\end{itemize}

In view of Case (a), a natural question that interest us is \emph{how to characterize the monodromy data $(r,s)$}.
For the Lam\'{e} equation (\ref{Lame}), Wang and the
third author \cite{LW2} proved the existence of a \emph{pre-modular form} $Z^n_{r,s}(\tau)$
such that the monodromy matrices $\rho(\ell_j)$'s of (\ref{Lame}) at $\tau=\tau_0$ are given by \eqref{Mono-1} for some $B$ if and only if $Z^n_{r,s}(\tau_0)=0$. This $Z^n_{r,s}(\tau)$ is holomorphic in $\tau$ if $(r,s)\in\mathbb{R}
^2\setminus \frac{1}{2}\mathbb{Z}^2$. Moreover, $Z^n_{r,s}(\tau)$ is a
modular form of weight $\frac{1}{2}n(n+1)$ with respect to the principal congruence subgroup $$\Gamma(m):=\{\gamma
\in SL(2,\mathbb{Z})|\gamma\equiv I_2\mod m \}$$ if $(r,s)$ is a $m$-torsion point; see \cite{LW2}. Thus
$Z^n_{r,s}(\tau)$ is called a \emph{pre-modular form} in \cite{LW2}.

One main purpose of this paper is to extend the above
result in \cite{LW2} to include the Trebich-Verdier potential. Here is our first main result.

\begin{theorem}
\label{thm-premodular}  There exists a pre-modular form $Z_{r,s}^{\mathbf{n}%
}(\tau)$ defined in $\tau\in\mathbb{H}$ for any pair $(r,s)\in\mathbb{C}%
^2\setminus \frac{1}{2}\mathbb{Z}^2$ such that the followings hold.

\begin{enumerate}
\item[(a)] If $(r,s)=(\frac{k_1}{m},\frac{k_2}{m})$ with $m\in 2\mathbb{N}
_{\geq 2}$, $k_1,k_2\in\mathbb{Z}_{\geq 0}$ and $\gcd(k_1,k_2,m)=1$, then $
Z_{r,s}^{\mathbf{n}}(\tau)$ is a modular form of weight $
\sum_{k=0}^3n_k(n_k+1)/2$ with respect to  $
\Gamma(m)$.

\item[(b)] For $(r,s)\in\mathbb{C}^2\setminus\frac{1}{2}\mathbb{Z}^2$ and $\tau_0\in\mathbb{H}$ such that $r+s\tau_0\notin \Lambda_{\tau_0}$, $Z_{r,s}^{\mathbf{n}}(\tau_0)=0$ if and only
if there is $B\in\mathbb{C}$ such that H$(\mathbf{n}, B, \tau_0)$
has its monodromy matrices $\rho(\ell_1)$ and $\rho(\ell_2)$ given by \eqref{Mono-1}.
\end{enumerate}
\end{theorem}

To explain our construction of the pre-modular form, we have to briefly recall the hyperelliptic curve associated with H$(\mathbf{n}, B, \tau)$. It is classical that the product of two solutions of H$(\mathbf{n}, B, \tau)$ solves the second symmetric product equation of H$(\mathbf{n}, B, \tau)$:
\begin{equation}  \label{eq2--3}
\Phi^{\prime \prime \prime }(z;B)-4I_{\mathbf{n}}(z; B, \tau)\Phi^{\prime
}(z;B)-2I_{\mathbf{n}}^{\prime }(z; B, \tau)\Phi(z;B)=0.
\end{equation}
It is known (see e.g. \cite{Takemura1}) that (\ref{eq2--3}) has a unique even elliptic solution (still denoted by $\Phi(z;B)$).
Multiplying $\Phi$ and integrating \eqref{eq2--3}, we see that
\[
Q_{\mathbf{n}}(B;\tau):=\Phi{^{\prime }}(z;B)^2-2\Phi(z;B)\Phi^{\prime \prime }(z;B)+4I_{\mathbf{n}}(z; B, \tau)\Phi(z;B)^2
\]
is independent of $z$. The fact that the Treibich-Verdier potential is an
algebro-geometric solution of the KdV hierarchy follows from that $Q_{\mathbf{n}}(B;\tau)$ is a monic polynomial of $B$ up to a multiplicity of $\Phi(z;B)$; see \cite{GW}.
In this case, $Q_{\mathbf{n}}(B; \tau)$ is known as the \emph{spectral polynomial}
and $\Gamma_{\mathbf{n}}(\tau):=\{(B,W)|W^2=Q_{\mathbf{n}}(B;\tau)\}$ is called the
\emph{spectral curve} of the Treibich-Verdier potential.

In Part I \cite{CKL1}, we proved that the spectral curve $\Gamma_{\mathbf{n}%
}(\tau)$ can be embedded into Sym$^N E_{\tau}:=E_{\tau}^N/S_N$, where $N:=\sum_{k=0}^3n_k$.
Since Sym$^NE_\tau$ has a natural addition map to $E_\tau:
\{a_1,\cdots,a_N\}\mapsto\sum_{i=1}^Na_i-\sum_{k=1}^{3}n_k\frac{\omega_k}{2}$, the composition gives arise to
a finite morphism $\sigma_{\mathbf{n}}(\cdot|\tau):\overline{\Gamma_{\mathbf{
n}}(\tau)}\rightarrow E_\tau$, still called the \emph{addition map}. See Section \ref{Ex-premodular} for a brief overview. The main result of Part I \cite{CKL1} is to determine the degree of $\sigma_{\mathbf{n}}$.

\begin{theorem}
\label{thm-degree} \cite{CKL1} The addition map $\sigma_{\mathbf{n}}: \overline{\Gamma_{\mathbf{n}}(\tau)} \to E_\tau$
has degree $\sum_{k=0}^3n_k(n_k+1)/2$.
\end{theorem}

As we will see, Theorem \ref{thm-degree} determines the weight of the pre-modular form $Z_{r,s}^{\mathbf{n}}(\tau)$ in Theorem \ref{thm-premodular}.
After Theorem \ref{thm-degree}, the field $K(\overline{\Gamma_{\mathbf{n}}(\tau)})$ of rational functions on $\overline{\Gamma_{\mathbf{n}}(\tau)}$ is a finite extension over $K(E_{\tau})$ of degree $\sum_{k=0}^3n_k(n_k+1)/2$ via this addition map (or covering map) $\sigma_{\mathbf{n}}$. The second step is to find the primitive generator of this extension, for which we need to prove the uniqueness of H$(\mathbf{n}, B, \tau)$ with respect to the monodromy data in the completely reducible case; see Lemma \ref{lem2}.

Now as an application of Theorem \ref{thm-premodular}, we consider the following Liouville equation with \textit{four
singular sources}:%
\begin{equation}
\Delta u+e^{u}=8\pi \sum_{k=0}^{3}n_{k}\delta_{\frac{\omega_{k}}{2}}\quad\text{ on
}\; E_{\tau}, \label{mean}%
\end{equation}
where $\delta_{\omega_{k}/2}$ is the Dirac measure at $\frac{\omega_{k}}{2}$.
Not surprisingly, (\ref{mean}) is related to various research areas. Geometrically, a solution $u$ of (\ref{mean}) leads to a metric
$g=\frac{1}{2}e^{u}(dx^{2}+dy^{2})$ with constant Gaussian curvature $+1$
acquiring \textit{conic singularities at $\frac{\omega_{k}}{2}$'s}. It also
appears in statistical physics as the equation for the \textit{mean field
limit} of the Euler flow in Onsager's vortex model (cf. \cite{CLMP}), hence
also called a \textit{mean field equation}. Recently (\ref{mean}) was shown to be related to the
self-dual condensates of the Chern-Simons-Higgs model in superconductivity.
We refer the readers to \cite{EG,LY,NT1} and references therein for
recent developments of related subjects of (\ref{mean}).

The existence of solutions of (\ref{mean}) is very challenging from
the PDE point of view. In fact, \emph{the solvability of (\ref{mean})
essentially depends on the moduli $\tau$ in a sophisticated manner}. This
phenomenon was first discovered by Wang and the third author \cite{LW} when
they studied the case $n_{0}=1$ and $n_{1}=n_{2}=n_{3}=0$, i.e.
\begin{equation}
\Delta u+e^{u}=8\pi \delta_{0}\  \  \text{on}\ E_{\tau}. \label{eq1-1}%
\end{equation}
Among other things, they proved that
\begin{itemize}
\item if $\tau \in i\mathbb{R}_{>0}$, i.e. $E_{\tau}$
is a rectangular torus, then (\ref{eq1-1}) has \textit{no} solution;

\item
if $\tau=\frac{1}{2}+\frac{\sqrt{3}}{2}i$, i.e. $E_{\tau}$ is a rhombus
torus, then (\ref{eq1-1}) has solutions.
\end{itemize}
Recently, (\ref{eq1-1})
has been thoroughly investigated in \cite{CKLW,LW4}.

Therefore, a natural question is {\it how to give a precise characterization of those $\tau$'s such that (\ref{mean}) has solutions on such $E_{\tau}$}. Here we give an answer to this question in the even solution case.

\begin{theorem}\label{thm-mfe}
The mean field equation (\ref{mean}) has even solutions on $E_{\tau}$ if and only if there is $(r,s)\in\mathbb{R}^2\setminus\frac{1}{2}\mathbb{Z}^2$ such that $\tau$ is a zero of this pre-modular form $Z_{r,s}^{\mathbf{n}}(\cdot)$, i.e. $Z_{r,s}^{\mathbf{n}}(\tau)=0$.
\end{theorem}

Theorem \ref{thm-mfe} generalizes the result in \cite{LW2} where the Lam\'{e} case $n_{1}=n_{2}=n_{3}=0$ was studied. In this case, Wang and the third author \cite{LW2} also proved that {\it once $\Delta u+e^{u}=8 n_0\pi\delta_{0}$ on $E_{\tau}$ has solutions, then it also has an even solution}. We believe that this statement should also hold for (\ref{mean}) with general $n_k\in\mathbb{Z}_{\geq 0}$, which seems challenging and remains open.

Now let us consider two special cases of (\ref{mean}):
\begin{equation}\label{mfe-16pi}\Delta u+e^u=16\pi\delta_{0} \  \  \text{on}\ E_{\tau}\end{equation}
and
\begin{equation}\label{mfe-24pi}\Delta u+e^u=8\pi\sum_{k=1}^3\delta_{\frac{\omega_k}{2}} \  \  \text{on}\ E_{\tau}.\end{equation}
There seems no obvious relations between these two equations. Therefore, the following result is quite surprising from the PDE aspect.
\begin{theorem}\label{16pi-24pi}
The mean field equations (\ref{mfe-16pi}) and (\ref{mfe-24pi})
 has the same number of even solutions.
\end{theorem}

The rest of the paper is organized as follows. In Section \ref{Ex-premodular}, we recall the theory concerning H$(\mathbf{n},B,\tau)$ from Part I \cite{CKL1} and then give the proof of Theorem \ref{thm-premodular}. In Section \ref{app-mfe}, we apply Theorem \ref{thm-premodular} to prove Theorem \ref{thm-mfe}. We will also prove a general result which contains Theorem \ref{16pi-24pi} as a special case. Appendix A is devoted to the proof of Theorem \ref{thm4-111} which is needed in the construction of the pre-modular form.

\section{Existence of pre-modular forms}

\label{Ex-premodular}

The purpose of this section is to construct the pre-modular form and prove Theorem \ref{thm-premodular}.
First we recall some basic theory concerning H$(\mathbf{n},B,\tau)$ from Part I \cite{CKL1}. As mentioned before, by changing variable $z\to z+\frac{\omega_k}{2}$ if necessary, we always assume $n_0\geq 1$.

(i) Any solution of $H({\bf n},B,\tau)$ is meromorphic in $\mathbb{C}$. The corresponding second symmetric product equation
\[\Phi^{\prime \prime \prime }(z;B)-4I_{\mathbf{n}}(z;B,\tau)\Phi^{\prime
}(z;B)-2I_{\mathbf{n}}^{\prime }(z;B,\tau)\Phi(z;B)=0\]
has a {\it unique} even elliptic solution $\Phi_{e}(z;B)$ expressed by
\begin{equation}
\Phi_{e}(z;B)=C_{0}(B)  +\sum_{k=0}^{3}\sum
_{j=0}^{n_{k}-1}b_{j}^{(k)}(B)\wp(z+\tfrac{\omega_{k}}{2})^{n_{k}-j}
\label{3rd1}%
\end{equation}
where $C_{0}(B), b_{j}^{(k)}(B)$ are all polynomials in $B$ with $\deg C_{0}>\max_{j,k} \deg b_{j}^{(k)}$ and the leading coefficient of $C_{0}(B)$
being $\frac{1}{2}$. Moreover, $\Phi_{e}(z;B)=y_1(z;B)$ $y_1(-z;B)$, where $y_1(z;B)$ is a common eigenfunction of the monodromy matrices of $H({\bf n},B,\tau)$ and, up to a constant, can be written as
\begin{equation}\label{yby}y_1(z;B)=y_{\boldsymbol{a}}(z)
:=e^{c({\boldsymbol{a}})z}\frac{\prod_{i=1}^{N}\sigma(z-a_i)}{\prod_{k=0}^3\sigma(z-\frac{\omega_k}{2})^{n_k}}\end{equation}
with some $\boldsymbol{a}=(a_1,\cdots, a_N)$ and $c(\boldsymbol{a})\in\mathbb{C}$. See (\ref{61-381}) for the expression of $c(\boldsymbol{a})$ in the completely reducible case.
From (\ref{yby}) and the transformation law (let $\eta_3=2\zeta(\frac{\omega_3}{2})=\eta_1+\eta_2$)
\[\sigma(z+\omega_k)=-e^{\eta_k(z+\frac{\omega_k}{2})}\sigma(z),\quad k=1,2,3,\] it is easy to see that
$y_1(-z;B)=y_{-\boldsymbol{a}}(z)$ up to a sign $(-1)^{n_1+n_2+n_3}$.

(ii) Let $W$ be the Wroskian of $y_1(z;B)$ and $y_1(-z;B)$, then $W^2=Q_{\bf n}(B;\tau)$, where
$$Q_{\bf n}(B;\tau):=\Phi_{e}'(z;B)^2-2\Phi_{e}(z;B)\Phi_{e}^{\prime \prime }(z;B)+4I_{\mathbf{n}%
}(z;B,\tau)\Phi_{e}(z;B)^2$$
is a monic polynomial in $B$ with \emph{odd degree} and independent of $z$. Define the hyperelliptic curve $\Gamma_{\bf n}(\tau)$ by
\[\Gamma_{\bf n}(\tau):=\{(B, W)\,|\,W^2=Q_{\bf n}(B;\tau)\}.\]
Then the map $i_{\bf n}: \Gamma_{\bf n}(\tau)\to \text{Sym}^{N} E_{\tau}:=E_{\tau}^N/S_N$ defined by
\[i_{\bf n}(B, W):=\{[a_1],\cdots, [a_{N}]\}\]
is an embedding, where $\{[a_1],\cdots, [a_{N}]\}$ is uniquely determined by $y_1(z;B)$ via (\ref{yby}). Since $-W$ be the Wroskian of $y_1(-z;B)=y_{-{\boldsymbol{a}}}(z)$ and $y_1(z;B)$, we have
\[i_{\bf n}(B, -W)=\{-[a_1],\cdots, -[a_{N}]\}.\]

(iii) The monodromy of H$(\mathbf{n}, B, \tau)$ is completely reducible if and only if $y_1(z;B)=y_{\boldsymbol{a}}(z)$ and $y_1(-z;B)=y_{-\boldsymbol{a}}(z)$ are linearly independent, which is also equivalent to
\begin{equation}\label{a-0a}\{[a_1],\cdots,[a_N]\}\cap \{-[a_1],\cdots,-[a_N]\}=\emptyset.
\end{equation}
In this case, since $a_j\neq 0$ in $E_{\tau}$ for all $j$ and $n_0\not =0$, we have
\begin{equation}
c(\pm\boldsymbol{a})=\sum_{i=1}^{N}\zeta(\pm a_{i})-\sum_{k=1}^{3}\frac{n_{k}\eta
_{k}}{2}, \label{61-381}%
\end{equation}
which follows by inserting (\ref{yby}) into H$(\mathbf{n},B,\tau)$ and computing the leading terms at the singularity $0$; see Theorem \ref{thm4.1--1}.
Besides, the $(r,s)$ defined by
\begin{equation}
\left \{
\begin{array}
[c]{l}%
\sum_{i=1}^{N}a_{i}-\sum_{k=1}^{3}\frac{n_{k}\omega_{k}}{2}=r+s\tau \\
\sum_{i=1}^{N}\zeta(a_{i})  -\sum_{k=1}^{3}\frac{n_{k}\eta_{k}
}{2}=r\eta_{1}+s\eta_{2}
\end{array}
\right. \label{rs2}%
\end{equation}
satisfies $(r,s)\notin \frac{1}{2}\mathbb{Z}^2$. Furthermore, with respect to $y_{\boldsymbol{a}}(z)$ and $y_{-\boldsymbol{a}}(z)$, the monodromy matrices are given by
\begin{equation}
\rho(\ell_{1})=%
\begin{pmatrix}
e^{-2\pi is} & 0\\
0 & e^{2\pi is}%
\end{pmatrix}
,\text{ }\rho(\ell_{2})=%
\begin{pmatrix}
e^{2\pi ir} & 0\\
0 & e^{-2\pi ir}%
\end{pmatrix}.
\label{61.3512}%
\end{equation}

(iv) Let $Y_{{\bf n}}(\tau)$ be the image of $\Gamma_{{\bf n}}(\tau)$ in Sym$^{N}E_\tau$
under $i_{{\bf n}}$, i.e.
\begin{equation}
Y_{\mathbf{n}}(\tau)  =\left \{
\begin{array}
[c]{r}%
[\boldsymbol{a}]=\{[a_{1}],\cdot \cdot \cdot,[a_{N}]\}\text{ }\in $Sym$^NE_{\tau}|\text{
}y_{\boldsymbol{a}}(z)  \text{ defined in }\\
\text{(\ref{yby}) is a solution of H}(\mathbf{n},B,\tau) \text{ for some } B
\end{array}
\right \}, \label{set}%
\end{equation}
and $X_{{\bf n}}(\tau)$ be the image of $\{(B,W)\in\Gamma_{{\bf n}}|W\neq
0\}$ under $i_{{\bf n}}$, i.e.
\begin{equation}
X_{\mathbf{n}}(\tau)  = \{[\boldsymbol{a}]\in Y_{\mathbf{n}}(\tau)| \text{ (\ref{a-0a}) holds}\}. \label{set11}%
\end{equation}
Clearly $Y_{\mathbf{n}}(\tau)\setminus X_{\mathbf{n}}(\tau)$ consists of those finite branch points, i.e. those $\boldsymbol{a}$'s such that $y_{\boldsymbol{a}}(z)$ and $y_{-\boldsymbol{a}}(z)$ are linearly dependent, which is equivalent to
\begin{equation}\label{bran}\{[a_1],\cdots,[a_N]\}=\{-[a_1],\cdots,-[a_N]\}.\end{equation}
The number of finite branch points is at most $\deg Q_{\mathbf{n}}(B)$. Besides,
it was proved in Part I \cite{CKL1} that $\overline{X_{\mathbf{n}}(  \tau )  }=\overline
{Y_{\mathbf{n}}(  \tau )  }=Y_{\mathbf{n}}(\tau)
\cup\{\infty_0\}  $, where
\begin{equation}\label{infty0}
\infty_0:=\bigg(  \overset{n_{0}}{\overbrace{0,\cdot \cdot \cdot,0}}%
,\overset{n_{1}}{\overbrace{\tfrac{\omega_{1}}{2},\cdot \cdot \cdot,\tfrac
{\omega_{1}}{2}}},\overset{n_{2}}{\overbrace{\tfrac{\omega_{2}}{2},\cdot
\cdot \cdot,\tfrac{\omega_{2}}{2}}},\overset{n_{3}}{\overbrace{\tfrac{\omega_{3}%
}{2},\cdot \cdot \cdot,\tfrac{\omega_{3}}{2}}}\bigg)  .
\end{equation}

(v)
The first formula of (\ref{rs2}) motivates us to study the addition map $\sigma
_{\mathbf{n}}:\overline{X_{\mathbf{n}}(  \tau )  }\rightarrow E_{\tau}$ (also called a covering map in \cite[Section 4]{Takemura4}):
\[
\sigma_{\mathbf{n}}( [\boldsymbol{a}]) : =\sum_{i=1}^{N}[
a_{i}]  -\sum_{k=1}^{3}[\tfrac{n_{k}\omega_{k}}{2}].%
\]
Since $2\sum_{k=1}^{3}[\tfrac{n_{k}\omega_{k}}{2}]=[0]$, we have
\[
\sigma_{\mathbf{n}}( [-\boldsymbol{a}])  =-\sum_{i=1}^{N}[
a_{i}]  -\sum_{k=1}^{3}[\tfrac{n_{k}\omega_{k}}{2}]=-\sigma_{\mathbf{n}}( [\boldsymbol{a}]).%
\]
Since the algebraic curve $\overline{X_{\mathbf{n}}(\tau)}$ is irreducible, $\sigma_{\mathbf{n}}$ is a finite morphism and $\deg\sigma_{\mathbf{n}}$ is well-defined. Theorem \ref{thm-degree} says that
\begin{equation}\label{degreef}
\deg\sigma_{\mathbf{n}}=\frac{1}{2}\sum_{k=0}^3 n_k(n_k+1).
\end{equation}

The above theories can be found in Part I \cite{CKL1}. Here we also need the following result, which will give a precise characterization of $X_{\mathbf{n}}(\tau)$.
\begin{theorem}
\label{thm4-111}Suppose $\boldsymbol{a}=\{a_{1},\cdot \cdot \cdot
,a_{N}\}$ satisfies%
\begin{equation}
[  a_{i}]  \not \in E_{\tau}[2],\text{ }[a_{i}]
\not =\pm [  a_{j}]  ,\text{ }\forall i\not =j. \label{I-062}%
\end{equation}
Then $y_{\boldsymbol{a}}(z)$ is a solution of H$(\mathbf{n}, B, \tau)$ for some $B$
if and only if $\boldsymbol{a}$ satisfies\begin{equation}
\sum_{i=1}^{N}\wp^{\prime}(a_{i})\wp(a_{i})^{l}=0\text{ for }0\leq l\leq
n_{0}-2, \label{I-050}%
\end{equation}%
\begin{equation}
\sum_{i=1}^{N}\wp^{\prime}(a_{i})\prod_{j=1,\neq i}^{N}(\wp(a_{j})-e_{k})^{l}=0\text{
for }1\leq l\leq n_{k},\text{ }k\in \{1,2,3\}. \label{I-051}%
\end{equation}
\end{theorem}

Theorem \ref{thm4-111} in the Lam\'{e} case $n_1=n_2=n_3=0$ was proved in \cite{CLW}.
The proof of Theorem \ref{thm4-111} for general $n_k$ is technical and long, and will be given in Appendix A. Note that if $[\boldsymbol{a}]\in X_{\mathbf{n}}(\tau)$, then (\ref{a-0a}) implies $[a_i]\notin E_{\tau}[2]$ and $[a_i]\neq -[a_j]$ for all $i,j$, so $a_i$ is a zero of $y_{\boldsymbol{a}}(z)$ which must be simple, i.e. $[a_i]\neq [a_j]$ for any $i\neq j$ and so (\ref{I-062}) holds. In conclusion,
\begin{equation}
X_{\mathbf{n}}(\tau)  = \{[\boldsymbol{a}]\in \text{Sym}^NE_{\tau}| \text{$\boldsymbol{a}$ satisfies (\ref{I-062})-(\ref{I-051})}\}. \label{set111}
\end{equation}

Now we proceed to construct a pre-modular form $Z_{r,s}^{\mathbf{n}}(\tau)$. Let $K(E_{\tau})$ and $K(\overline{X_{\mathbf{n}}(\tau)})$ be the field of rational functions on $E_{\tau}$ and $\overline{X_{\mathbf{n}}(\tau)}$, respectively. Then (\ref{degreef}) indicates that $K(\overline{X_{\mathbf{n}}(\tau)})$ is a finite extension over $K(E_{\tau})$ and
\begin{equation}\label{finite-ext}\left[K(\overline{X_{\mathbf{n}}(\tau)}): K(E_{\tau})\right]=\deg\sigma_{\mathbf{n}}=\frac{1}{2}\sum_{k=0}^3 n_k(n_k+1).\end{equation}
A basic question is \emph{how to find a primitive generator}?

Motivated by \cite{LW2} and (\ref{rs2}), we consider the function
\[\mathbf{z}_{\mathbf{n}}(a_1,\cdots,a_N):=\zeta\Bigg(\sum_{i=1}^N a_i-\sum_{k=1}^{3}\frac{n_{k}\omega_{k}}{2}\Bigg)-\sum_{i=1}^{N}\zeta(a_{i})  +\sum_{k=1}^{3}\frac{n_{k}\eta_{k}
}{2},\]
which is meromorphic and periodic in each $a_i$ and hence defines a rational function on $E_{\tau}^N$. By symmetry, it descends to a rational function on Sym$^N E_{\tau}$. We denote the restriction $\mathbf{z}_{\mathbf{n}}|_{\overline{X_{\mathbf{n}}(\tau)}}$ also by $\mathbf{z}_{\mathbf{n}}$, which is a rational function on $\overline{X_{\mathbf{n}}(\tau)}$.

\begin{lemma}\label{lemma-pole}
The poles of $\mathbf{z}_{\mathbf{n}}$ on $\overline{X_{\mathbf{n}}(\tau)}$ are precisely the fiber $\sigma_{\mathbf{n}}^{-1}([0])$.
\end{lemma}

\begin{proof}
Fix any $\boldsymbol{a}=\{a_1,\cdots,a_N\}\in \overline{X_{\mathbf{n}}(\tau)}\setminus \{\infty_0\}= Y_{\mathbf{n}}(\tau)$. It suffices to prove that
\begin{equation}\label{neqinf}
\sum_{i=1}^{N}\zeta(a_{i})\neq \infty.
\end{equation}

If $\boldsymbol{a}\in X_{\mathbf{n}}(\tau)$, i.e.
not a branch point, then (\ref{neqinf}) follows from (\ref{61-381}). So it suffices to consider that $\boldsymbol{a}\in Y_{\mathbf{n}}(\tau)\setminus X_{\mathbf{n}}(\tau)$ is a finite branch point. Then it might happen that $a_{i}=0$ for some $i$'s.
Since the number of branch points is finite, we can take a sequence $X_{\mathbf{n}}(\tau)\ni\boldsymbol{a}^m\to \boldsymbol{a}$. Denote $\boldsymbol{a}^m=\{a^m_1,\cdots,a^m_N\}$. Note from (\ref{set}) that $y_{\boldsymbol{a}^m}(z)$ given by (\ref{yby}) is a solution of H$(\mathbf{n}, B_m,\tau)$ for some $B_m\in\mathbb{C}$. Since $\boldsymbol{a}^m\to \boldsymbol{a}\neq \infty_0$, we proved in Part I \cite{CKL1} that $B_m$ are uniformly bounded and so do $c(\boldsymbol{a}^m)$. Consequently, we see from (\ref{61-381}) that
\[\sum_{i=1}^{N}\zeta(a_{i})=\lim_{m\to\infty}\sum_{i=1}^{N}\zeta(a_{i}^m)\neq\infty.\]
The proof is complete.
\end{proof}

\begin{lemma}
\label{lem2}Let $\boldsymbol{a}, \boldsymbol{b}\in X_{\mathbf{n}}(\tau)$ be not branch points. Suppose
\begin{equation}
\sigma_{\mathbf{n}}( \boldsymbol{a})  =\sigma_{\mathbf{n}}(
\boldsymbol{b})  \text{ \ and \ }\mathbf{z}_{\mathbf{n}}(
\boldsymbol{a})  =\mathbf{z}_{\mathbf{n}}(\boldsymbol{b})  .
\label{rs3}%
\end{equation}
Then $\boldsymbol{a}=\boldsymbol{b}$.
\end{lemma}

\begin{proof}
Under our assumption (\ref{rs3}), we can take $(a_1,\cdots,a_N), (b_1,\cdots,b_N)\in\mathbb{C}^N$ to be representatives of $\boldsymbol{a}, \boldsymbol{b}$ such that
\begin{equation}\label{a=b}\sum_{i=1}^N a_i=\sum_{i=1}^N b_i,\quad\sum_{i=1}^{N}\zeta(a_{i})=\sum_{i=1}^{N}\zeta(b_{i}).\end{equation}
By (\ref{set}), there exist $B_1, B_2$ such that $y_{\boldsymbol{a}}(z)$ (resp. $y_{\boldsymbol{b}}(z)$) given by (\ref{yby}) is a solution of H$(\mathbf{n}, B_1,\tau)$ (resp. H$(\mathbf{n}, B_2,\tau)$). Then (\ref{rs2}), (\ref{61.3512}) and (\ref{a=b}) imply that H$(\mathbf{n}, B_1,\tau)$ and H$(\mathbf{n}, B_2,\tau)$) have the same global monodromy data $(r,s)\notin\frac{1}{2}\mathbb{Z}^2$, namely $y_{\boldsymbol{a}}(z)$ and $y_{\boldsymbol{b}}(z)$, which are solutions of H$(\mathbf{n}, B_1,\tau)$ and H$(\mathbf{n}, B_2,\tau)$ respectively, satisfy the same transformation law:
\begin{equation}\label{same-trans}y(z+\omega_1)=e^{-2\pi i s}y(z),\quad y(z+\omega_2)=e^{2\pi i r}y(z).\end{equation}

Now we use the following interesting observation from \cite[Lemma 3.5]{LW2}: Denote $I=\sum_{k=0}^3n_k(n_k+1)\wp(z+\frac{\omega_k}{2})$ and $I_j=I+B_j$ for $j=1,2$. Suppose $w_j''=I_jw_j$ for $j=1,2$. Then $w_1w_2$ satisfies the following forth order ODE:
\begin{equation}\label{forth-order}
q''''-2(I_1+I_2)q''-6I'q'+((B_1-B_2)^2-2I'')q=0.
\end{equation}
This statement can be proved by direct computations.
Furthermore, it is easy to see that the local exponents of (\ref{forth-order}) at $\frac{\omega_k}{2}$ are
\begin{equation}\label{loc-ex}-2n_k,\quad 1,\quad 3,\quad 2n_k+2.\end{equation}
Recalling $y_{-\boldsymbol{a}}(z)=(-1)^{n_1+n_2+n_3}y_{\boldsymbol{a}}(-z)$ and $y_{-\boldsymbol{b}}(z)=(-1)^{n_1+n_2+n_3}y_{\boldsymbol{b}}(-z)$, it follows from (\ref{same-trans}) that
\[q(z):=y_{\boldsymbol{a}}(z)y_{-\boldsymbol{b}}(z)-y_{-\boldsymbol{a}}(z)y_{\boldsymbol{b}}(z)\]
is \emph{an odd elliptic solution} of (\ref{forth-order}). Consequently, (\ref{loc-ex}) infers that $\frac{\omega_k}{2}$ must be a zero of $q(z)$ (with order $1$ or $3$) for any $k$. This implies that $q(z)$ has no poles and so $q(z)\equiv 0$, i.e. $y_{\boldsymbol{a}}(z)y_{-\boldsymbol{b}}(z)$ is even. This implies the zero set $\boldsymbol{a}\cup (-\boldsymbol{b})=(-\boldsymbol{a})\cup \boldsymbol{b}$, and it follows from (\ref{a-0a}) (i.e. $[a_i]\neq -[a_j]$, $[b_i]\neq -[b_j]$ for any $i, j$) that $\boldsymbol{a}=\boldsymbol{b}$. Clearly this also infers $B_1=B_2$.
\end{proof}

\begin{theorem}
\label{thm9}There is a weighted homogeneous polynomial
\begin{equation}\label{minipol}
W_{\mathbf{n}}(\mathbf{z})  \in \mathbb{Q}[e_{1}(\tau), e_2(\tau), e_3(\tau), \wp(\sigma|\tau)
,\wp^{\prime}(\sigma|\tau)][\mathbf{z}]
\end{equation}
of $\mathbf{z}$-degree $d_{\mathbf{n}}=\deg \sigma_{\mathbf{n}}$ such that for $\sigma=\sigma_{\mathbf{n}}(\boldsymbol{a})$, we have
\begin{equation}
W_{\mathbf{n}}(\mathbf{z}_{\mathbf{n}})(\boldsymbol{a})=0. \label{zeroo}%
\end{equation}
Here, the weights of $\mathbf{z}$, $\wp(\sigma)$, $e_k$'s, $\wp'(\sigma)$ are $1$, $2$, $2$, $3$ respectively.

Indeed, $\mathbf{z}_{\mathbf{n}}(\boldsymbol{a})$ is a primitive generator of the finite extension of rational function field $K(\overline{X_{\mathbf{n}}(\tau)})$ over $K(E_{\tau})$ with $W_{\mathbf{n}}(\mathbf{z})$ being its minimal polynomial.
\end{theorem}

\begin{proof} Thanks to Lemmas \ref{lemma-pole}-\ref{lem2}, the proof is similar to \cite[Theorem 3.2]{LW2}. Here we give the proof for completeness.

Recall (\ref{finite-ext}). Since $\mathbf{z}_{\mathbf{n}}\in K(  \overline{X_{\mathbf{n}}(\tau)})$, its minimal polynomial $W_{\mathbf{n}}(\mathbf{z})\in K(E_{\tau})[\mathbf{z}]$ exists with degree
$d_{\mathbf{n}}:=\deg W_{\mathbf{n}}|\deg \sigma_{\mathbf{n}}$.

Note that if $\boldsymbol{a}$ is a branch point, then it follows from (\ref{bran}) that $\sigma_{\mathbf{n}}(\boldsymbol{a})\in E_{\tau}[2]$.
To prove $d_{\mathbf{n}}=\deg \sigma_{\mathbf{n}}$, i.e. $\mathbf{z}_{\mathbf{n}}(\boldsymbol{a})$ is a primitive generator, we take $\sigma_0\in E_{\tau}\setminus E_{\tau}[2]$ outside the branch loci of $\sigma_{\mathbf{n}}: \overline{X_{\mathbf{n}}(\tau)}\to E_{\tau}$. Then there are precisely $\deg \sigma_{\mathbf{n}}$ different points $\boldsymbol{a}\in X_{\mathbf{n}}(\tau)$ with $\sigma_{\mathbf{n}}(\boldsymbol{a})=\sigma_0$, and Lemma \ref{lem2} shows that these $\deg \sigma_{\mathbf{n}}$ different points $\boldsymbol{a}$ give $\deg \sigma_{\mathbf{n}}$ different values $\mathbf{z}_{\mathbf{n}}(\boldsymbol{a})$. Therefore, for $\sigma=\sigma_0$, the polynomial $W_{\mathbf{n}}(\mathbf{z})$ of degree $d_{\mathbf{n}}|\deg \sigma_{\mathbf{n}}$ has $\deg \sigma_{\mathbf{n}}$ distinct zeros, which implies $d_{\mathbf{n}}=\deg \sigma_{\mathbf{n}}$.

Since Lemma \ref{lemma-pole} shows that $\mathbf{z}_{\mathbf{n}}$ has no poles over $E_{\tau}^{\times}:=E_{\tau}\setminus\{[0]\}$, it is indeed {\it integral} over the affine Weierstrass model of $E_{\tau}^{\times}$ with the coordinate ring (let $x=\wp(\sigma)$, $y=\wp'(\sigma)$)
\[R(E_{\tau}^{\times})=\mathbb{C}[x, y]/(y^2-4(x-e_1)(x-e_2)(x-e_3)),\]
i.e. the minimal polynomial $W_{\mathbf{n}}$ is monic in $R(E_{\tau}^{\times})[\mathbf{z}]$.
This implies (\ref{minipol}) and the homogeneity of $W_{\mathbf{n}}$, where as in \cite{LW2}, the coefficients lie in $\mathbb{Q}$, instead of just in $\mathbb{C}$, follows from standard elimination theory and two facts: (i) The equations (\ref{I-050})-(\ref{I-051}) of $X_{\mathbf{n}}(\tau)$ (see (\ref{set111})) are defined over $\mathbb{Q}[e_1,e_2,e_3]$, and (ii) the addition map $E_{\tau}^N\to E_{\tau}$ is defined over $\mathbb{Q}$ which, together with the addition formulas of elliptic functions, infers that
\begin{align*}
\mathbf{z}_{\mathbf{n}}(\boldsymbol{a})&=\zeta\Bigg(\sum_{i=1}^N a_i-\sum_{k=1}^{3}\frac{n_{k}\omega_{k}}{2}\Bigg)-\sum_{i=1}^{N}\zeta(a_{i})  +\sum_{k=1}^{3}\frac{n_{k}\eta_{k}
}{2}\\
&=\sum_{j_1<\cdots<j_m;\, m\, \text{odd}} f^{(1)}_{j_1,\cdots, j_m}(\wp(a_{j_1}),\cdots,\wp(a_{j_m}))\prod_{i=1}^m\wp'(a_{j_i}),
\end{align*}
\begin{align*}
\wp(\sigma_{\mathbf{n}}( \boldsymbol{a}))&=\wp\Bigg(\sum_{i=1}^N a_i-\sum_{k=1}^{3}\frac{n_{k}\omega_{k}}{2}\Bigg)\\
&=\sum_{j_1<\cdots<j_m;\, m\, \text{even}} f^{(2)}_{j_1,\cdots, j_m}(\wp(a_{j_1}),\cdots,\wp(a_{j_m}))\prod_{i=1}^m\wp'(a_{j_i}),
\end{align*}
\begin{align*}
\wp'(\sigma_{\mathbf{n}}( \boldsymbol{a}))
=\sum_{j_1<\cdots<j_m;\, m\, \text{odd}} f^{(3)}_{j_1,\cdots, j_m}(\wp(a_{j_1}),\cdots,\wp(a_{j_m}))\prod_{i=1}^m\wp'(a_{j_i}),
\end{align*}
where $f^{(k)}_{j_1,\cdots, j_m}(x_1,\cdots,x_m)\in \mathbb{Q}(e_1,e_2,e_3)(x_1,\cdots,x_m)$. The minimal polynomial $W_{\mathbf{n}}$ is obtained by eliminating the terms $\wp(a_{j}), \wp'(a_j)$'s via these formulas and the equations (\ref{I-050})-(\ref{I-051}) of $X_{\mathbf{n}}(\tau)$.
The proof is complete.
\end{proof}

As in \cite{LW2}, for any $(r,s)\in \mathbb{C}^2\backslash \frac{1}{2}\mathbb{Z}^{2}$, we define
\begin{equation}
Z_{r,s}(\tau)  :=\zeta(r+s\tau|\tau)  -r\eta_{1}(\tau)-s\eta_{2}(\tau), \label{zrs}%
\end{equation}
(note $Z_{r,s}(\tau)\equiv \infty$ if $(r,s)\in \mathbb{Z}^2$ and $Z_{r,s}(\tau)\equiv 0$ if $(r,s)\in \frac{1}{2}\mathbb{Z}^{2}\setminus\mathbb{Z}^2$) and
\begin{equation}
Z^{\mathbf{n}}_{r,s}(\tau)  :=W_{\mathbf{n}}(Z_{r,s})(r+s\tau;\tau), \label{zrs1}%
\end{equation}
i.e. by letting $\sigma=r+s\tau$ and $\mathbf{z}=Z_{r,s}(\tau)$ in (\ref{minipol}). Clearly this $Z^{\mathbf{n}}_{r,s}(\tau)$ is holomorphic in $\tau$ for given $(r,s)\in\mathbb{R}^2\setminus\frac{1}{2}\mathbb{Z}^2$. We show that this $Z^{\mathbf{n}}_{r,s}(\tau)$ is precisely the pre-modular form in Theorem \ref{thm-premodular}.

\begin{theorem}[=Theorem \ref{thm-premodular}]
\label{thm-premodular-copy}  Let $Z_{r,s}^{\mathbf{n}}(\tau)$ be defined in (\ref{zrs1}). Then the followings hold.

\begin{enumerate}
\item[(a)] $Z_{r,s}^{\mathbf{n}}(\tau)$ is a pre-modular form in the sense that, if $(r,s)=(\frac{k_1}{m},\frac{k_2}{m})$ with $m\in 2\mathbb{N}%
_{\geq 2}$, $k_1,k_2\in\mathbb{Z}_{\geq 0}$ and $\gcd(k_1,k_2,m)=1$, then $%
Z_{r,s}^{\mathbf{n}}(\tau)$ is a modular form of weight $\sum_{k=0}^3n_k(n_k+1)/2$ with respect to the principal congruence subgroup $%
\Gamma(m)$.

\item[(b)] For $(r,s)\in\mathbb{C}^2\setminus\frac{1}{2}\mathbb{Z}^2$ and $\tau_0\in\mathbb{H}$ such that $r+s\tau_0\notin \Lambda_{\tau_0}$, $Z_{r,s}^{\mathbf{n}}(\tau_0)=0$ if and only
if there is $B\in\mathbb{C}$ such that H$(\mathbf{n}, B, \tau_0)$
has its monodromy matrices $\rho(\ell_1)$ and $\rho(\ell_2)$ given by \eqref{61.3512}.
\end{enumerate}
\end{theorem}

\begin{proof} (a) It is well-known that $e_k(\tau)$, $k=1,2,3$, are all modular forms of weight $2$ with respect to $\Gamma(2)$. Since $Z_{r,s}(\tau)$, $\wp(r+s\tau|\tau)$ and $\wp'(r+s\tau|\tau)$ are modular forms of weight $1$, $2$, $3$ respectively, with respect to $\Gamma(m)$ for any $m$-torsion point $(r,s)=(\frac{k_1}{m},\frac{k_2}{m})$, $m\geq 3$, the assertion (a) follows directly from the homogeneity of $W_{\mathbf{n}}$ in Theorem \ref{thm9}.

(b) First we prove the sufficient part. Suppose for some H$(\mathbf{n},B,\tau_0)$, its monodromy matrices $\rho(\ell_1)$ and $\rho(\ell_2)$ are given by \eqref{61.3512}, i.e. the monodromy is completely reducible.
Then there exists $\boldsymbol{a}\in X_{\mathbf{n}}(\tau_0)$ such that $y_{\boldsymbol{a}}(z)$ is a solution of H$(\mathbf{n},B,\tau_0)$ and (\ref{rs2}) holds, i.e.
\begin{equation}\label{rrs}\left \{
\begin{array}
[c]{l}%
\sum_{i=1}^{N}a_{i}-\sum_{k=1}^{3}\frac{n_{k}\omega_{k}}{2}=r+s\tau_0 \\
\sum_{i=1}^{N}\zeta(a_{i})  -\sum_{k=1}^{3}\frac{n_{k}\eta_{k}
}{2}=r\eta_{1}+s\eta_{2}.
\end{array}
\right.\end{equation}
From here we have $\sigma_{\mathbf{n}}(\boldsymbol{a})=[r+s\tau_0]\neq [0]$ and
\begin{align*}
\mathbf{z}_{\mathbf{n}}(\boldsymbol{a})=\zeta(r+s\tau_0|\tau_0)
-r\eta_1(\tau_0)-s\eta_2(\tau_0)=Z_{r,s}(\tau_0).
\end{align*}
Then it follows from (\ref{zeroo}) and (\ref{zrs1}) that $Z_{r,s}^{\mathbf{n}}(\tau_0)=0$.

Conversely, suppose $Z_{r,s}^{\mathbf{n}}(\tau_0)=0$. Then it follows from (\ref{zrs1}) that for $\sigma=r+s\tau_0\notin \Lambda_{\tau_0}$, $Z_{r,s}(\tau_0)$ is a zero of $W_{\mathbf{n}}(\mathbf{z})$. This, together with Theorem \ref{thm9}, implies the existence of $\boldsymbol{a}\in \overline{X_{\mathbf{n}}(\tau_0)}\setminus\{\infty_0\}=Y_{\mathbf{n}}(\tau_0)$ such that
\[\sum_{i=1}^{N}a_{i}-\sum_{k=1}^{3}\frac{n_{k}\omega_{k}}{2}=\sigma=r+s\tau_0,\quad \mathbf{z}_{\mathbf{n}}(\boldsymbol{a})=Z_{r,s}(\tau_0),\]
which is equivalent to (\ref{rrs}). The definition (\ref{set}) of $Y_{\mathbf{n}}(\tau_0)$ yields that $y_{\boldsymbol{a}}(z)$ is a solution of some H$(\mathbf{n}, B, \tau_0)$. By (\ref{rrs}), we see that with respect to $y_{\boldsymbol{a}}(z)$ and $y_{-\boldsymbol{a}}(z)$, the monodromy matrices $\rho(\ell_1)$ and $\rho(\ell_2)$ are given by \eqref{61.3512}.

The proof is complete.
\end{proof}

We conclude this section by the following remark.

After Theorem \ref{thm-premodular}, a further question arises: {\it What are the explicit expressions of these pre-modular forms}? This question is very difficult because the weight $\frac{1}{2}\sum n_k(n_k+1)$ is large for general $\mathbf{n}$.
It is known from \cite{Dahmen,LW2} that (write $Z=Z_{r,s}(\tau)$, $\wp=\wp(r+s\tau|\tau)$ and $\wp^{\prime
}=\wp^{\prime}(r+s\tau|\tau)$ for convenience):
\[
Z_{r,s}^{(1,0,0,0)}=Z,\quad Z_{r,s}^{(2,0,0,0)}=Z^{3}-3\wp Z-\wp^{\prime},
\]%
\begin{align*}
Z_{r,s}^{(3,0,0,0)}=  &  Z^{6}-15\wp Z^{4}-20\wp^{\prime}Z^{3}+\left(
\tfrac{27}{4}g_{2}-45\wp^{2}\right)  Z^{2}\\
&  -12\wp \wp^{\prime}Z-\tfrac{5}{4}(\wp^{\prime})^{2}.
\end{align*}
{\allowdisplaybreaks%
\begin{align*}
Z_{r,s}^{(4,0,0,0)}=  &  Z^{10}-45\wp Z^{8}-120\wp^{\prime}Z^{7}+(\tfrac
{399}{4}g_{2}-630\wp^{2})Z^{6}-504\wp \wp^{\prime}Z^{5}\\
&  -\tfrac{15}{4}(280\wp^{3}-49g_{2}\wp-115g_{3})Z^{4}+15(11g_{2}-24\wp
^{2})\wp^{\prime}Z^{3} \\
&  -\tfrac{9}{4}(140\wp^{4}-245g_{2}\wp^{2}+190g_{3}\wp+21g_{2}^{2}%
)Z^{2}\label{z-n-4}\\
&  -(40\wp^{3}-163g_{2}\wp+125g_{3})\wp^{\prime}Z+\tfrac{3}{4}(25g_{2}%
-3\wp^{2})(\wp^{\prime})^{2}.
\end{align*}
}%
The above formulas are all for the Lam\'{e} case.
For $n\geq 5$, the explicit expression of $Z_{r,s}^{(n,0,0,0)}(\tau)$ is not known so far. See \cite{CKLW,Dahmen,LW2} for applications of the above formulas of $Z_{r,s}^{(n,0,0,0)}(\tau)$, $n\leq 4$.

Here are new examples of $Z_{r,s}^{\bf n}(\tau)$ for the Treibich-Verdier potential case:
\[Z_{r,s}^{(1,1,0,0)}=Z^2-\wp+e_1,\]
\[Z_{r,s}^{(1,0,1,0)}=Z^2-\wp+e_2,\quad Z_{r,s}^{(1,0,0,1)}=Z^2-\wp+e_3,\]
\[Z_{r,s}^{(2,1,0,0)}=Z^4+3(e_1-2\wp)Z^2-4\wp'Z-3(\wp^2+e_1\wp+e_1^2-\tfrac{g_2}{4}),\]
and similarly, the expression of $Z_{r,s}^{(2,0,1,0)}$ (resp. $Z_{r,s}^{(2,0,0,1)}$) is obtained by replacing $e_1$ in $Z_{r,s}^{(2,1,0,0)}$ with $e_2$ (resp. $e_3$).
The proof of these new formulas will be given in a forthcoming work, where we will also study further properties of these pre-modular forms, such as the following interesting formula:
\[Z_{r,s}^{(1,1,0,0)}(\tau)=4Z_{r,\frac{s}{2}}(2\tau)Z_{r,\frac{s+1}{2}}(2\tau).\]
We believe that these new formulas will have interesting applications.

\section{Application to the mean field equation}

\label{app-mfe}

This section is devoted to the proof of Theorems \ref{thm-mfe} and \ref{16pi-24pi}. Recall from (\ref{Mono-1})-(\ref{Mono-31}) that the monodromy group of H$(\mathbf{n}, B, \tau)$ is conjugate to a subgroup of
$SU(2)$, i.e. the monodromy of H$(\mathbf{n}, B, \tau)$ is {\it unitary}, if and only if Case (a) occurs with some $(r,s)\in\mathbb{R}^2\setminus\frac{1}{2}\mathbb{Z}^2$. Applying Theorem \ref{thm-premodular}, this is equivalent to that $\tau$ is a zero of $Z_{r,s}^{\mathbf{n}}(\cdot)$ for some $(r,s)\in\mathbb{R}^2\setminus\frac{1}{2}\mathbb{Z}^2$. Therefore, Theorem \ref{thm-mfe} is a direct corollary of Theorem \ref{thm-premodular} and the following result.

\begin{theorem}\label{thm-3-1}
The mean field equation \begin{equation}
\Delta u+e^{u}=8\pi \sum_{k=0}^{3}n_{k}\delta_{\frac{\omega_{k}}{2}}\quad\text{ on
}\; E_{\tau}, \label{mean3}%
\end{equation} has an even solution if and only if there exists $B\in \mathbb{C}$ such
that the monodromy of  H$(\mathbf{n}, B, \tau)$ is unitary.

Furthermore, the number of even solutions equals to the number of those $B$'s such that the monodromy of H$(\mathbf{n}, B, \tau)$ is unitary.
\end{theorem}

For the Lam\'{e} case $n_1=n_2=n_3=0$, Theorem \ref{thm-3-1} was proved in \cite{CLW}. For general case $n_k\in\mathbb{Z}_{\geq 0}$ as considered here, the necessary part of the first assertion in Theorem \ref{thm-3-1} was already proved in \cite{CL-AJM,EG-2015}. Here we sketch the proof from \cite{CL-AJM} for later usage.

\begin{proof}[Proof of the necessary part of the first assertion in Theorem \ref{thm-3-1} \cite{CL-AJM}] Let $u(z)$ be a solution of (\ref{mean3}). Then
the Liouville theorem says that there
is a local meromorphic function $f(z)$ away from $E_{\tau}[2]=\{\frac{\omega_{k}}{2}\,|\,0\leq k\leq 3\}+\Lambda_{\tau}$ such that%
\begin{equation}
u(z)=\log \frac{8|f^{\prime}(z)|^{2}}{(1+|f(z)|^{2})^{2}}. \label{502}%
\end{equation}
This $f(z)$ is called a developing map. By differentiating (\ref{502}), we
have
\begin{equation}
u_{zz}-\frac{1}{2}u_{z}^{2}= \{ f;z \}:=\left(  \frac{f^{\prime \prime}%
}{f^{\prime}}\right)  ^{\prime}-\frac{1}{2}\left(  \frac{f^{\prime \prime}%
}{f^{\prime}}\right)  ^{2}. \label{new22}%
\end{equation}
Conventionally, the RHS of this identity is called the Schwarzian derivative
of $f(z)$, denoted by $\{ f;z \}$. Note that outside
the singularities $E_{\tau}[2]$,
\[
\left(  u_{zz}-\tfrac{1}{2}u_{z}^{2}\right)  _{\bar{z}}=\left(  u_{z\bar{z}%
}\right)  _{z}-u_{z}u_{z\bar{z}}=-\tfrac{1}{4}\left(  e^{u}\right)
_{z}+\tfrac{1}{4}e^{u}u_{z}=0.
\]
Combining this with the local behavior of $u(z)$ at $\frac{\omega_k}{2}$: $u(z)=4n_k\ln|z-\frac{\omega_k}{2}|+O(1)$,
we conclude that $u_{zz}-\frac{1}{2}u_{z}^{2}$ is an \emph{elliptic function} with at most \emph{double poles} at $E_\tau[2]$.

Now suppose $u(z)$ is even, i.e. $u(z)=u(-z)$. Then
\begin{equation}
u_{zz}-\frac{1}{2}u_{z}^{2}=-2\bigg[  \sum_{k=0}^{3}n_{k}(n_{k}+1)\wp
(z+\tfrac{\omega_{k}}{2}|\tau)+B\bigg]  =-2I_{\mathbf{n}}(z;B,\tau) \label{cc-1}%
\end{equation}
for some constant $B=B(u)$, because due to the evenness, $u_{zz}-\frac{1}%
{2}u_{z}^{2}$ has no residues at $z\in E_{\tau}[2]$.
Since $\{f;z\}=-2I_{\mathbf{n}}(z;B,\tau)$, a classical result says that there are linearly independent solutions $y_1(z), y_2(z)$ of H$(\mathbf{n}, B, \tau)$ such that
\begin{equation}\label{fy}f(z)=\frac{y_1(z)}{y_2(z)}.\end{equation}
Recalling Section 2 that $y_1(z), y_2(z)$ are both meromorphic, we see that the developing map
$f(z)$ is single-valued near each $\frac{\omega_{k}}{2}$ and
then can be extended to be an entire meromorphic function in $\mathbb{C}$.

Define the Wronskian
\[W:=y_{1}'(z)y_2(z)-y_1(z)y_2'(z).\]
Then $W$ is a nonzero constant. By inserting (\ref{fy}) into (\ref{502}), a direct computation leads to
\[2\sqrt{2}W e^{-\frac{1}{2}u(z)}=|y_1(z)|^2+|y_2(z)|^2.\]
Since $u(z)$ is single-valued and doubly periodic, we immediately see that the monodromy group with respect to $(y_1(z), y_2(z))$ is contained in $SU(2)$, namely the monodromy of H$(\mathbf{n}, B, \tau)$ is unitary.\end{proof}

The main task of this section is to prove the sufficient part and the $1-1$ correspondence between even solutions and those $B$'s such that the monodromy of H$(\mathbf{n}, B, \tau)$ is unitary. The following proof is also not difficult by applying the monodromy theory of H$(\mathbf{n}, B, \tau)$ in Section \ref{Ex-premodular}.

\begin{proof}[Proof of Theorem \ref{thm-3-1}]
First we prove the sufficient part of the first assertion.

Suppose there is $B\in \mathbb{C}$ such
that the monodromy of  H$(\mathbf{n}, B, \tau)$ is unitary. Then by (\ref{Mono-1}) and (\ref{a-0a})-(\ref{61.3512}), there exist $(r,s)\in\mathbb{R}^2\setminus\frac{1}{2}\mathbb{Z}^2$ and $\boldsymbol{a}\in Y_{\mathbf{n}}(\tau)$ such that
\[y_{\pm\boldsymbol{a}}(z)=e^{c(\pm{\boldsymbol{a}})z}\frac{\prod_{i=1}^{N}\sigma(z\mp a_i)}{\prod_{k=0}^3\sigma(z-\frac{\omega_k}{2})^{n_k}}\] are linearly independent solutions of  H$(\mathbf{n}, B, \tau)$ and
\begin{equation}
\begin{pmatrix}
y_{\boldsymbol{a}}(z+\omega_1)\\
y_{-\boldsymbol{a}}(z+\omega_1)
\end{pmatrix}
=\left(
\begin{matrix}
e^{-2\pi is} & 0\\
0 & e^{2\pi is}%
\end{matrix}
\right)
\begin{pmatrix}
y_{\boldsymbol{a}}(z)\\
y_{-\boldsymbol{a}}(z)
\end{pmatrix}
, \label{61-34}%
\end{equation}%
\[
\begin{pmatrix}
y_{\boldsymbol{a}}(z+\omega_2)\\
y_{-\boldsymbol{a}}(z+\omega_2)
\end{pmatrix}
=\left(
\begin{matrix}
e^{2\pi ir} & 0\\
0 & e^{-2\pi ir}%
\end{matrix}
\right)
\begin{pmatrix}
y_{\boldsymbol{a}}(z)\\
y_{-\boldsymbol{a}}(z)
\end{pmatrix}
.
\]
Now we define
\begin{equation}\label{deve-f}f(z):=\frac{y_{\boldsymbol{a}}(z)}{y_{-\boldsymbol{a}}(z)}=e^{2z\sum_{i=1}^{N} \zeta(a_i)}\frac{\prod_{i=1}^{N}\sigma(z- a_i)}{\prod_{i=1}^{N}\sigma(z+a_i)}\end{equation}
and
\[u(z):=\log \frac{8|f^{\prime}(z)|^{2}}{(1+|f(z)|^{2})^{2}}.\]
We claim that this $u(z)$ is an even solution of (\ref{mean3}).

Clearly (\ref{61-34}) and $(r,s)\in\mathbb{R}^2$ yield that $u(z)$ is doubly periodic and hence well-defined on $E_{\tau}$. Furthermore, $f(-z)=\frac{1}{f(z)}$ infers that $u(z)=u(-z)$.

Since (\ref{a-0a}), (\ref{set11}) and (\ref{set111}) imply that $a_i\notin E_{\tau}[2]$ and $a_i$'s (resp. $-a_i$'s) are all simple zeros (resp. simple poles) of $f(z)$, we have that (i) $f(z)$ are holomorphic at $\frac{\omega_k}{2}$ for all $k$, and (ii) $u(\pm a_i)\neq \infty$ for all $i$ and so
\[\Delta u+e^u=0\quad \text{on}\; E_{\tau}\setminus\{z\in E_{\tau}\setminus \cup_i\{\pm a_i\}\,|\, f'(z)=0\}.\]
Since $f(z)=\frac{y_{\boldsymbol{a}}(z)}{y_{-\boldsymbol{a}}(z)}$ gives
\begin{align*}\left(  \frac{f^{\prime \prime}%
}{f^{\prime}}\right)  ^{\prime}-\frac{1}{2}\left(  \frac{f^{\prime \prime}%
}{f^{\prime}}\right)  ^{2}&=-2I_{\mathbf{n}}(z;B,\tau)\\
&=-2\bigg[  \sum_{k=0}^{3}n_{k}(n_{k}+1)\wp
(z+\tfrac{\omega_{k}}{2}|\tau)+B\bigg],\end{align*}
we easily conclude that $f'(z)\neq 0$ for any $z\notin E_{\tau}[2]$ and ord$_{z=\frac{\omega_k}{2}}f'(z)=2n_k$, i.e. $u(z)=4n_k\ln|z-\frac{\omega_k}{2}|+O(1)$ near $\frac{\omega_k}{2}$.
In conclusion, $u(z)$ is an even solution of (\ref{mean3}).

Next we prove that the even solution is unique for this given $B$. Suppose $\tilde{u}(z)$ is an even solution of (\ref{mean3}) corresponding to the same $B$ with $u(z)$, i.e. (\ref{cc-1}) holds. Our goal is to prove $\tilde{u}(z)=u(z)$.

By the proof of the necessary part of Theorem \ref{thm-3-1}, we may let $\hat{f}(z)=\frac{y_1(z)}{y_2(z)}$ be a developing map of $\tilde{u}(z)$ such that $y_1(z), y_2(z)$ are linearly independent solutions of H$(\mathbf{n}, B, \tau)$ and satisfy
\[\begin{pmatrix}
y_{1}(z+\omega_j)\\
y_{2}(z+\omega_j)
\end{pmatrix}
=M_j
\begin{pmatrix}
y_{1}(z)\\
y_{2}(z)
\end{pmatrix}
,\quad M_j\in SU(2),\quad j=1,2.\]
Since $M_1M_2=M_2M_1$, there is a matrix $P=\begin{pmatrix}a&b\\c&d\end{pmatrix}\in SU(2)$ such that
$\tilde{M}_j:=PM_jP^{-1}$ are both diagonal matrices for $j=1,2$.
Define
\[\begin{pmatrix}
\tilde{y}_{1}(z)\\
\tilde{y}_{2}(z)
\end{pmatrix}
:=P
\begin{pmatrix}
y_{1}(z)\\
y_{2}(z)
\end{pmatrix}\]
to be another pair of linearly independent solutions of H$(\mathbf{n}, B, \tau)$ and
\[\tilde{f}(z):=\frac{\tilde{y}_{1}(z)}{\tilde{y}_{2}(z)
}=\frac{a\hat{f}(z)+b}{c\hat{f}(z)+d}=:P\hat{f}(z).\]
The fact $P\in SU(2)$ implies that $\tilde{f}(z)$ is also a developing map of $\tilde{u}(z)$, i.e.
\begin{equation}\label{u-f-tilde}\tilde{u}(z)=\log \frac{8|\hat{f}^{\prime}(z)|^{2}}{(1+|\hat{f}(z)|^{2})^{2}}=\log \frac{8|\tilde{f}^{\prime}(z)|^{2}}{(1+|\tilde{f}(z)|^{2})^{2}}.\end{equation}
Clearly
\[\begin{pmatrix}
\tilde{y}_{1}(z+\omega_j)\\
\tilde{y}_{2}(z+\omega_j)
\end{pmatrix}
=\tilde{M}_j
\begin{pmatrix}
\tilde{y}_{1}(z)\\
\tilde{y}_{2}(z)
\end{pmatrix},\quad \tilde{M}_j\text{ is diagonal for } j=1,2.\]
Together with (\ref{61-34}), $(r,s)\notin \frac{1}{2}\mathbb{Z}^2$ and the fact that $\tilde{y}_{j}(z)$'s are linear combinations of $y_{\pm\boldsymbol{a}}(z)$, we easily conclude that, by reordering $\tilde{y}_{1}(z),
\tilde{y}_{2}(z)$ if necessary,
\[\tilde{y}_{1}(z)=c_1y_{\boldsymbol{a}}(z),\quad \tilde{y}_{2}(z)=c_2y_{-\boldsymbol{a}}(z)\]
for some constants $c_1, c_2\neq0$. So (\ref{deve-f}) gives $\tilde{f}(z)=cf(z)$ for $c=c_1/c_2\neq 0$ and then
\[\tilde{f}(-z)=cf(-z)=\frac{c}{f(z)}=\frac{c^2}{\tilde{f}(z)}.\]
Inserting this into (\ref{u-f-tilde}), it follows from $\tilde{u}(z)=\tilde{u}(-z)$ that $|c|=1$. In conclusion, $\tilde{f}(z)=cf(z)$ with $|c|=1$, which clearly infers that $\tilde{u}(z)=u(z)$. This proves the uniqueness of the even solution with respect to the given $B$. Therefore, the number of even solutions equals to the number of those $B$'s such that the monodromy of H$(\mathbf{n}, B, \tau)$ is unitary.

The proof is complete.
\end{proof}

As an application of Theorem \ref{thm-3-1}, we have the following result, which contains Theorem \ref{16pi-24pi} as a special case.

\begin{theorem}\label{thm-npi}
Given $n\in\mathbb{N}$, we define
\[n_0:=\frac{n}{2}-1,\;n_1=n_2=n_3:=\frac{n}{2}\quad\text{if $n$ is even},\]
\[n_0:=\frac{n+1}{2},\;n_1=n_2=n_3:=\frac{n-1}{2}\quad\text{if $n$ is odd}.\]
Then the mean field equations
\begin{equation}\label{mfe-8npi}\Delta u+e^u=8\pi n\delta_{0} \  \  \text{on}\ E_{\tau}\end{equation}
and
\begin{equation}\label{mfe-8npi1}\Delta u+e^u=8\pi\sum_{k=0}^3 n_k\delta_{\frac{\omega_k}{2}} \  \  \text{on}\ E_{\tau}\end{equation}
has the same number of even solutions.
\end{theorem}
\begin{proof}
It was proved by Takemura \cite[Section 4]{Takemura5} that H$((n,0,0,0), B, \tau)$ and H$ ((n_0,n_1,n_2,n_3), B, \tau)$ are isomonodromic (i.e. their monodromy representations are the same) for any $(B,\tau)$. Therefore, this theorem follows from Theorem \ref{thm-3-1}.
\end{proof}

In Theorem \ref{thm-npi}, the case $n=1$ is trivial, and the first nontrivial case $n=2$ gives Theorem \ref{16pi-24pi}, which has the following consequence for rhombus tori.

\begin{corollary}\label{corollary-1}
There exists $b^{\ast}\in(\frac{\sqrt{3}}{2},\frac{6}
{5})$ such that for $\tau=\frac{1}{2}+ib$ with $b>b^*$, (\ref{mfe-24pi}) on $E_{\tau}$ has even solutions; while for $\tau=\frac{1}{2}+i\frac{\sqrt{3}}{2}$, (\ref{mfe-24pi}) on $E_{\tau}$ has no even solutions.
\end{corollary}

\begin{proof} It was proved in \cite[Theorem A.2]{CL-JDG} that there exists $b^{\ast}\in(\frac{\sqrt{3}}{2},\frac{6}
{5})$ such that for $\tau=\frac{1}{2}+ib$ with $b>b^*$, (\ref{mfe-16pi}) on $E_{\tau}$ has even solutions. Furthermore, we proved in \cite[Theorem 3.1]{CKL-CAG} that (\ref{mfe-16pi}) on $E_{\tau}$ has no solutions for $\tau=\frac{1}{2}+i\frac{\sqrt{3}}{2}$. Therefore, this assertion follows from Theorem \ref{16pi-24pi}.
\end{proof}

\begin{remark}In Theorem \ref{thm-npi}, it is easy to see that
$n(n+1)=\sum_{k=0}^3 n_k(n_k+1)$, namely the pre-modular forms $Z_{r,s}^{(n,0,0,0)}(\tau)$ and $Z_{r,s}^{(n_0,n_1,n_2,n_3)}(\tau)$ have the same weight. Theorem \ref{thm-npi} strongly suggests $Z_{r,s}^{(n,0,0,0)}(\tau)=Z_{r,s}^{(n_0,n_1,n_2,n_3)}(\tau)$, which will be studied in a future work.
\end{remark}

\begin{remark}
Theorem \ref{thm-npi} and Corollary \ref{corollary-1} indicate that the mean field equations with multiple singularities
\begin{equation}\label{mfe-multiple}\Delta u+e^u=8\pi\sum n_k\delta_{p_k} \  \  \text{on}\ E_{\tau}\end{equation}
might be studied by establishing relations with some other mean field equations with less singularities. We will apply this idea to study (\ref{mfe-multiple}) in future works.
\end{remark}

\appendix

\section{Proof of Theorem \ref{thm4-111}}

\label{Floquet-theory}

This appendix is devoted to the proof of Theorem \ref{thm4-111}. First we need
the following result which was essentially proved in \cite{GW1,Takemura1}.

\begin{theorem}
\label{thm4.1--1}Suppose $\boldsymbol{a}=\{a_{1},\cdot \cdot \cdot
,a_{N}\}$ satisfies%
\begin{equation}
[a_{i}]  \not \in E_{\tau}[2],\text{ }[  a_{i}]
\not =[  a_{j}]  ,\text{ }\forall i\not =j. \label{I-62}%
\end{equation}
Then \[y_{\boldsymbol{a}}(z)=e^{c({\boldsymbol{a}})z}\frac{\prod_{i=1}^{N}\sigma(z- a_i)}{\prod_{k=0}^3\sigma(z-\frac{\omega_k}{2})^{n_k}}\] is a solution to H$(\mathbf{n},B,\tau)$ with some $B$
if and only if $\boldsymbol{a}$ satisfies
{\allowdisplaybreaks
\begin{align}
&  \sum_{k=1}^{3}n_{k}\left[  \zeta(a_{i}+\tfrac{\omega_{k}}{2})+\zeta
(a_{i}-\tfrac{\omega_{k}}{2})-2\zeta(a_{i})\right] \label{I-63}\\
=  &  2\sum_{j\not =i}^{N}\left[  \zeta(a_{i}-a_{j})+\zeta(a_{j}%
)-\zeta(a_{i})\right]  ,\text{ \ }1\leq i\leq N,\nonumber
\end{align}
}{\allowdisplaybreaks%
\begin{align}
&  \sum_{i=1}^{N}\left(  \zeta(a_{i}+\tfrac{\omega_{l}}{2})+\zeta
(a_{i}-\tfrac{\omega_{l}}{2})\right) \label{I-64}\\
&  =2\sum_{i=1}^{N}\zeta(a_{i}),\text{ whenever }n_{l}\not =%
0,l\in \{1,2,3\},\nonumber
\end{align}
}and $c(\boldsymbol{a})$, $B$ are determined by
\begin{equation}
c(\boldsymbol{a})=\sum_{i=1}^{N}\zeta
(a_{i})-\frac{1}{2}\sum_{k=1}^{3}n_{k}\eta_{k}, \label{I-65}%
\end{equation}%
\begin{equation}
B=(2n_{0}-1)\sum_{i=1}^{N}\wp(a_{i})-\sum_{k=1}^{3}n_{k}(n_{k}%
+2n_{0})e_{k}. \label{I-66}%
\end{equation}

\end{theorem}

\begin{proof}
We sketch the proof here for the reader's convenience.
Note that
\begin{align*}
\frac{y_{\boldsymbol{a}}'(z)}{y_{\boldsymbol{a}}(z)}=c({\boldsymbol{a}})+\sum_{i=1}^N\zeta(z-a_i)
-\sum_{k=0}^3n_k\zeta(z-\tfrac{\omega_k}{2}),
\end{align*}
\begin{align*}
\left(\frac{y_{\boldsymbol{a}}'(z)}{y_{\boldsymbol{a}}(z)}\right)'=-\sum_{i=1}^N\wp(z-a_i)
+\sum_{k=0}^3n_k\wp(z-\tfrac{\omega_k}{2}),
\end{align*}
are both elliptic functions. Consider the elliptic function
\[h(z):=\left(\frac{y_{\boldsymbol{a}}'(z)}{y_{\boldsymbol{a}}(z)}\right)'
+\left(\frac{y_{\boldsymbol{a}}'(z)}{y_{\boldsymbol{a}}(z)}\right)^2-\sum_{k=0}^{3}n_{k}(n_{k}+1)\wp
(z+\tfrac{\omega_{k}}{2})-B.\]
Clearly $y_{\boldsymbol{a}}(z)$ is a solution of H$(\mathbf{n},B,\tau)$ if and only if $h(z)\equiv 0$ if and only if {\it none of $\frac{\omega_k}{2}$'s and $a_i$'s are poles of $h(z)$} and the constant term of the Laurent expansion at $z=0$ is $0$. By computing leading terms of the Laurent expansions at $z=\frac{\omega_k}{2}, a_i$, we easily obtain the conditions (\ref{I-63})-(\ref{I-66}).
\end{proof}

Clearly Theorem \ref{thm4-111} is a consequence of Theorem \ref{thm4.1--1} and the following result.

\begin{theorem}
\label{prop11}Suppose $\boldsymbol{a}=\{a_{1},\cdot \cdot \cdot
,a_{N}\}$ satisfies \begin{equation}
[  a_{i}]  \not \in E_{\tau}[2],\text{ }[a_{i}]
\not =\pm[  a_{j}]  ,\text{ }\forall i\not =j. \label{I-0620}%
\end{equation} Then (\ref{I-63})-(\ref{I-64}) are equivalent to
\begin{equation}
\sum_{i=1}^{N}\wp^{\prime}(a_{i})\wp(a_{i})^{l}=0\text{ for }0\leq l\leq
n_{0}-2, \label{I-0500}%
\end{equation}%
\begin{equation}
\sum_{i=1}^{N}\wp^{\prime}(a_{i})\prod_{j=1,\neq i}^{N}(\wp(a_{j})-e_{k})^{l}=0\text{
for }1\leq l\leq n_{k},\text{ }k\in \{1,2,3\}. \label{I-0510}%
\end{equation}
\end{theorem}
\begin{proof}
Define
\[f(z):=e^{2z\sum_{i=1}^{N} \zeta(a_i)}\frac{\prod_{i=1}^{N}\sigma(z- a_i)}{\prod_{i=1}^{N}\sigma(z+a_i)}.\]
First we claim that
\begin{equation}
\text{(\ref{I-63})-(\ref{I-64}) hold}\;\Longleftrightarrow\;\underset{z=\frac{\omega_{k}}{2}}{ord}f^{\prime}(z)=2n_{k}\text{ for all
}k. \label{kaxi}%
\end{equation}
With the help of Theorem \ref{thm4.1--1}, this claim can be proved by similar arguments as the sufficient part of Theorem \ref{thm-3-1}. We leave the details to the interested reader. Here we would like to give an elementary proof by using only the elliptic function theory but without using the ODE H$(\mathbf{n},B,\tau)$ and the Schwarzian derivative.

Recalling the addition formula
\begin{equation}
\zeta(u+v)+\zeta(u-v)-2\zeta(u)=\frac{\wp^{\prime}(u)}{\wp(u)-\wp(v)},
\label{81}%
\end{equation}
we
have{\allowdisplaybreaks%
\begin{align}
g(z)  & := \frac{f^{\prime}}{f}(z)=\sum_{i=1}^{N}\left(  \zeta
(z-a_{i})-\zeta(z+a_{i})+2\zeta(a_{i})\right) \label{800}\\
&  =\sum_{i=1}^{N}\frac{\wp^{\prime}(a_{i})}{\wp(z)-\wp(a_{i})}=\frac{\psi
(\wp(z))}{\prod_{i=1}^{N}(\wp(z)-\wp(a_{i}))},\nonumber
\end{align}
where}%
\begin{equation}
\psi(x):= \sum_{h=1}^{N}\wp^{\prime}(a_{h})\prod_{j\not =h}^{N}%
(x-\wp(a_{j})). \label{8000}%
\end{equation}
Clearly (\ref{I-64}) is equivalent to
\begin{equation}
\psi(e_{l})=g(\tfrac{\omega_l}{2})=0\text{ \ whenever }n_{l}\not =0\text{, }l=1,2,3.
\label{8002}%
\end{equation}

Notice from (\ref{I-0620}) that $g(z)$ is even elliptic with $2N$ simple poles $\pm a_{i},1\leq
i\leq N$. Since $f(\frac{\omega_{k}}{2})\not \in \{0,\infty \}$, it
is easy to see that ord$_{z=\frac{\omega_{k}}{2}}f^{\prime
}(z)=2n_{k}$ for all $k$ is
equivalent to saying that $\frac{\omega_{k}}{2}$ is a zero of $g(z)$ with
order $2n_{k}$ for all $k$ and so $g(z)$ has no other zeros (in particular, {\it $f'(z)$ has no other zeros}), namely
\begin{equation}
g(z)=d\frac{\prod_{k=1}^{3}(\wp(z)-e_{k})^{n_{k}}}{\prod
_{i=1}^{N}(\wp(z)-\wp(a_{i}))} \label{801}%
\end{equation}
for some constant $d\not =0$. For convenience, we define%
\begin{equation}
H(x):=  \prod_{k=1}^{3}(x-e_{k})^{n_{k}}.
\label{801-3}%
\end{equation}

First we prove the sufficient part of the claim (\ref{kaxi}). Suppose (\ref{801}) holds. Then
(\ref{8002}) gives (\ref{I-64}). By comparing (\ref{800}) and (\ref{801}%
)-(\ref{801-3}) we have%
\[
dH(x)=\psi(x)=\sum_{h=1}^{N}\wp^{\prime}(a_{h})\prod_{j\not =h}^{N}%
(x-\wp(a_{j})),\text{ \ }x=\wp(z).
\]
Taking derivative with respect to $x$ leads to{\allowdisplaybreaks%
\begin{equation}
dH^{\prime}(x)=\sum_{h=1}^{N}\wp^{\prime}(a_{h})\sum_{l\not =h}^{N}%
\prod_{j\not =h,l}^{N}(x-\wp(a_{j})),\text{ }\forall x\in \mathbb{C},
\label{803}%
\end{equation}
}where%
\begin{equation}
H^{\prime}(x)=H(x)\sum_{k=1}^{3}\frac{n_{k}%
}{x-e_{k}} . \label{803-2}%
\end{equation}
Now we fix any $i$. Then by letting $x=\wp(a_{i})$ in (\ref{803}), the RHS
becomes{\allowdisplaybreaks%
\begin{align*}
&  \sum_{h=1}^{N}\wp^{\prime}(a_{h})\sum_{l\not =h}^{N}\prod_{j\not =h,l}%
^{N}(\wp(a_{i})-\wp(a_{j}))\\
=  &  \wp^{\prime}(a_{i})\sum_{l\not =i}^{N}\prod_{j\not =i,l}^{N}(\wp
(a_{i})-\wp(a_{j}))+\sum_{h\not =i}^{N}\wp^{\prime}(a_{h})\sum_{l\not =h}%
^{N}\prod_{j\not =h,l}^{N}(\wp(a_{i})-\wp(a_{j}))\\
=  &  \wp^{\prime}(a_{i})\sum_{l\not =i}^{N}\prod_{j\not =i,l}^{N}(\wp
(a_{i})-\wp(a_{j}))+\sum_{h\not =i}^{N}\wp^{\prime}(a_{h})\prod_{j\not =%
h,i}^{N}(\wp(a_{i})-\wp(a_{j}))\\
=  &  \sum_{l\not =i}^{N}\left(  \wp^{\prime}(a_{i})+\wp^{\prime}%
(a_{l})\right)  \prod_{j\not =i,l}^{N}(\wp(a_{i})-\wp(a_{j})),
\end{align*}
}which gives
\begin{equation}
\frac{\sum_{h=1}^{N}\wp^{\prime}(a_{h})\sum_{l\not =h}^{N}\prod_{j\not =%
h,l}^{N}(\wp(a_{i})-\wp(a_{j}))}{\prod_{j\not =i}^{N}(\wp(a_{i})-\wp(a_{j}%
))}=\sum_{j\not =i}^{N}\frac{\wp^{\prime}(a_{i})+\wp^{\prime}(a_{j})}%
{\wp(a_{i})-\wp(a_{j})}. \label{804}%
\end{equation}
Together with (\ref{803}) and the addition formula%
\begin{equation}
\zeta(u+v)-\zeta(u)-\zeta(v)=\frac{1}{2}\frac{\wp^{\prime}(u)-\wp^{\prime}%
(v)}{\wp(u)-\wp(v)}, \label{622-1}%
\end{equation}
we obtain{\allowdisplaybreaks%
\begin{align}
\frac{dH^{\prime}(\wp(a_{i}))}{\prod_{j\not =i}^{N}(\wp(a_{i})-\wp(a_{j}))}
&  =\sum_{j\not =i}^{N}\frac{\wp^{\prime}(a_{i})+\wp^{\prime}(a_{j})}%
{\wp(a_{i})-\wp(a_{j})}\label{805}\\
&  =2\sum_{j\not =i}^{N}\left[  \zeta(a_{i}-a_{j})+\zeta(a_{j})-\zeta
(a_{i})\right]  .\nonumber
\end{align}
}By (\ref{800}) and (\ref{801})-(\ref{801-3}) again, we have%
\[
1=\underset{z=a_{i}}{\text{Res}}g(z)=\frac{1}{\wp^{\prime}(a_{i})}\frac
{dH(\wp(a_{i}))}{\prod_{j\not =i}^{N}(\wp(a_{i})-\wp(a_{j}))}.
\]
This, together with (\ref{805}), (\ref{803-2}) and the addition formula
(\ref{81}), gives{\allowdisplaybreaks%
\begin{align}
&  2\sum_{j\not =i}^{N}\left[  \zeta(a_{i}-a_{j})+\zeta(a_{j})-\zeta
(a_{i})\right] \label{806}\\
=  &  \frac{\wp^{\prime}(a_{i})H^{\prime}(\wp(a_{i}))}{H(\wp(a_{i}))}%
=\sum_{k=1}^{3}n_{k}\frac
{\wp^{\prime}(a_{i})}{\wp(a_{i})-e_{k}}\nonumber \\
=  &\sum_{k=1}^{3}n_{k}\left[  \zeta \left(  a_{i}+\tfrac{\omega_{k}}%
{2}\right)  +\zeta \left(  a_{i}-\tfrac{\omega_{k}}{2}\right)  -2\zeta
(a_{i})\right]  ,\nonumber
\end{align}
}namely (\ref{I-63}) holds.

Conversely, suppose (\ref{I-63})-(\ref{I-64}) hold, then (\ref{8002}) and
(\ref{806}) hold. By (\ref{804}), the second equality of (\ref{805}) and
(\ref{806}), it is easy to see that
\begin{equation}
\frac{\sum_{h=1}^{N}\wp^{\prime}(a_{h})\sum_{l\not =h}^{N}\prod_{j\not =%
h,l}^{N}(\wp(a_{i})-\wp(a_{j}))}{\prod_{j\not =i}^{N}(\wp(a_{i})-\wp(a_{j}%
))}\frac{H(\wp(a_{i}))}{\wp^{\prime}(a_{i})H^{\prime}(\wp(a_{i}))}=1
\label{807}%
\end{equation}
holds for all $i$. Define a polynomial%
\[
Q(x):= \psi^{\prime}(x)\cdot H(x)-\psi(x)\cdot H^{\prime}(x).
\]
Recalling (\ref{8000}), (\ref{8002}) and (\ref{801-3}), we easily obtain%
\[
\deg Q(x)\leq N-1+\sum_{k=1}^{3}n_{k}\text{ and }\left. \prod
_{k=1}^{3}(x-e_{k})^{n_{k}}\right \vert Q(x).
\]
Since (\ref{807}) just says $Q(\wp(a_{i}))=0$ for all $i$, we see that $Q(x)$
has at least $N+\sum_{k=1}^{3}n_{k}$ zeros, so $Q(x)\equiv0$. Consequently,
$\psi(x)=dH(x)$ for some constant $d\not =0$, i.e. (\ref{801}) holds, which
infers ord$_{z=\frac{\omega_{k}}{2}}f^{\prime}(z)=2n_{k}$ for all
$k$.
This proves the claim (\ref{kaxi}).

Thanks to (\ref{kaxi}), it suffices for us to prove
the equivalence between (\ref{I-0500})-(\ref{I-0510}) and ord$_{z=\frac{\omega_{k}}{2}}f^{\prime}(z)=2n_{k}$ for all
$k$. Clearly
(\ref{800}) gives%
\begin{align*}
\frac{f^{\prime}(z)}{f(z)}  &  =\sum_{i=1}^{N}\frac{\wp^{\prime}(a_{i})}%
{\wp(z)-\wp(a_{i})}=\frac{1}{\wp(z)}\sum_{i=1}^{N}\frac{\wp^{\prime}(a_{i}%
)}{1-\frac{\wp(a_{i})}{\wp(z)}}\\
&  =\sum_{l=0}^{\infty}\frac{\sum_{i=1}^{N}\wp^{\prime}(a_{i})\wp(a_{i})^{l}%
}{\wp(z)^{l+1}}.
\end{align*}
Therefore, ord$_{z=0}f^{\prime}(z)=2n_{0}$ if and only if
(\ref{I-0500}) holds. Similarly, for $k\in \{1,2,3\}$, by using the addition
formula%
\[
\wp(z-\tfrac{\omega_{k}}{2})-e_{k}=\frac{\mu_{k}}{\wp(z)-e_{k}},\quad
\mu_{k}=\frac{1}{2}\wp^{\prime \prime}(\tfrac{\omega_{k}}{2})\not =0,
\]
we easily obtain
{\allowdisplaybreaks
\begin{align*}
&  \frac{f^{\prime}(z)}{f(z)}-\sum_{i=1}^{N}\frac{\wp^{\prime}(a_{i})}%
{e_{k}-\wp(a_{i})}\\
=  &  \sum_{i=1}^{N}\left(  \frac{\wp^{\prime}(a_{i})}{\wp(z)-\wp(a_{i}%
)}+\frac{\wp^{\prime}(a_{i})}{\wp(a_{i})-e_{k}}\right)
=    \sum_{i=1}^{N}\frac{\wp^{\prime}(a_{i})}{(\wp(a_{i})-e_{k})(1-\frac
{\wp(a_{i})-e_{k}}{\wp(z)-e_{k}})}\\
=  &  \sum_{i=1}^{N}\frac{\wp^{\prime}(a_{i})}{(\wp(a_{i})-e_{k})(1-\frac
{\wp(a_{i})-e_{k}}{\mu_{k}}(\wp(z-\tfrac{\omega_{k}}{2})-e_{k}))}\\
=  &  -\mu_{k}\sum_{i=1}^{N}\frac{\wp^{\prime}(a_{i})}{(\wp(a_{i})-e_{k}%
)^{2}(\wp(z-\tfrac{\omega_{k}}{2})-e_{k})(1-\frac{\mu_{k}}{(\wp(a_{i}%
)-e_{k})(\wp(z-\tfrac{\omega_{k}}{2})-e_{k})})}\\
=  &  -\sum_{l=1}^{\infty}\frac{\mu_{k}^{l}\sum_{i=1}^{N}\frac{\wp^{\prime
}(a_{i})}{(\wp(a_{i})-e_{k})^{l+1}}}{(\wp(z-\tfrac{\omega_{k}}{2})-e_{k})^{l}%
}\\
=&-\sum_{l=1}^{\infty}\frac{\mu_{k}^{l}\frac{\sum_{i=1}^{N}\wp^{\prime}(a_{i})\prod_{j=1,\neq i}^{N}(\wp(a_{j})-e_{k})^{l+1}}{\prod_{i=1}^{N}(\wp(a_{i})-e_{k})^{l+1}}}{(\wp(z-\tfrac{\omega_{k}}{2})-e_{k})^{l}%
}.
\end{align*}
}Therefore, ord$_{z=\frac{\omega_{k}}{2}}f^{\prime}(z)=2n_{k}$ if
and only if (\ref{I-0510}) holds.

The proof is complete.
\end{proof}

\bigskip

{\bf Acknowledgements}
The research of the first author was supported by NSFC (No. 11701312).

\end{document}